\documentclass{amsart}

\usepackage{graphicx}
\usepackage{parskip}
\usepackage{amsthm,amsfonts,amsmath,amssymb,amscd}
\usepackage{tikz-cd}
\usepackage{hyperref}
\usepackage{verbatim}
\usepackage{enumerate}
\usepackage{color}
\usepackage{pinlabel}
\usepackage{mathtools}

\makeatletter
\newtheorem*{rep@theorem}{\rep@title}
\newcommand{\newreptheorem}[2]{%
\newenvironment{rep#1}[1]{%
 \def\rep@title{#2 \ref{##1}}%
 \begin{rep@theorem}}%
 {\end{rep@theorem}}}
\makeatother

\newtheorem{thm}{Theorem}[section]
\newreptheorem{thm}{Theorem}
\newtheorem{lma}[thm]{Lemma}
\newtheorem{cor}[thm]{Corollary}
\newtheorem{prp}[thm]{Proposition}

\theoremstyle{remark}
\newtheorem{rmk}[thm]{Remark}

\theoremstyle{definition}

\newtheorem{exm}[thm]{Example}
\newtheorem{quest}[thm]{Question}

\newcommand{\CC}{\mathbb{C}}

\newcommand{\R}{\mathbf{R}}
\newcommand{\T}{\mathbf{T}}
\newcommand{\RR}{\mathbb{R}}

\newcommand{\OP}{\operatorname}

\newcommand{\id}{\mathrm{id}}

\begingroup 
\makeatletter 
\@for\theoremstyle:=definition,remark,plain,TheoremNum\do{%
\expandafter\g@addto@macro\csname th@\theoremstyle\endcsname{%
\addtolength\thm@preskip\parskip 
}%
} 
\endgroup

\title[Exact Lagrangians inside symplectisations]{Exact Lagrangians in four-dimensional symplectisations}

\author{Georgios Dimitroglou Rizell}
\address{Department of Mathematics\\
Uppsala University\\
Box 480\\
SE-751 06 Uppsala\\
Sweden}
\email{georgios.dimitroglou@math.uu.se}

\thanks{The author is supported by the Knut and Alice Wallenberg Foundation under the grants KAW 2021.0191 and KAW 2021.0300, and by the Swedish Research Council under the grant number 2020-04426.}

\begin{document}
\begin{abstract}
In this note we provide explicit constructions of exact Lagrangian embeddings of tori and Klein bottles inside the symplectisation of an overtwisted contact three-manifold. Note that any closed exact Lagrangian in the symplectisation is displaceable by a Hamiltonian isotopy. We also use positive loops to exhibit elementary examples of topologically linked Legendrians that are dynamically non-interlinked in the sense of Entov--Polterovich.
\end{abstract}
\maketitle

\section{Introduction}

Both the existence and deformation of Lagrangian submanifolds in symplectic manifolds are known to exhibit different types of rigidity. In other words,there are subtle obstructions that can only be detected by techniques that gofar beyond classical topology. Notable such techniques and obstructions are 
based upon Gromov's theory of pseudoholomorphic curves \cite{Gromov}. This includes the theory of Floer homology, which was invented by Floer \cite{Floer}. The latter is a chain complex that is generated by Lagrangian intersection points for a pair of Lagrangians, whose differential counts pseudoholomorphic strips and, in favourable settings, the homology of this complex has been shown to be invariant under Hamiltonian isotopy.

We now recall a few classical examples of rigidity phenomena. First, there are several non-trivial obstructions to the existence of Lagrangian embeddings. Gromov's famous work \cite{Gromov} shows that any closed Lagrangian submanifold inside the standard symplectic vector space $\left(\CC^n,\omega_0=d\sum_i x_i\,dy_i\right)$ must admit a pseudoholomorphic disc with boundary on the submanifold; this is a relative 2-cycle of positive symplectic area. There are plenty of Lagrangian tori in any symplectic manifold, since they can be constructed e.g.~inside a symplectic Darboux chart. However, Gromov's result 
implies that none of them are exact. Viterbo later showed that also the Maslov class of such a Lagrangian satisfies strong restrictions, e.g.~it must be non-vanishing \cite{Viterbo:new}. More recently Cieliebak--Mohnke strengthened this by showing that there must exist a symplectic disc of Maslov index two with boundary on any Lagrangian torus in the symplectic vector space \cite{Cieliebak:Punctured}. Further, there are subtle obstructions on the topology of a Lagrangian in the standard symplectic vector space that have been derived using the technique of pseudoholomorphic curves. Notably, Shevchishin \cite{Shevchishin} and Nemirovski \cite{Nemirovski} have shown that the standard symplectic vector-space $(\CC^2,\omega_0)$ does not admit a Lagrangian embedding of the Klein bottle. The author together with Goodman--Ivrii showed that any Lagrangian torus inside the symplectic vector space must be unknotted \cite{Dimitroglou:Isotopy}.

The \emph{exact} Lagrangian submanifolds are known to satisfy particularly strong forms of rigidity, not only concerning their existence, but also concerning their deformations by Hamiltonian isotopies. Floer has shown that a closed exact Lagrangian submanifold of a symplectic manifold that has bounded geometry at infinity in the sense of Gromov \cite{Gromov} cannot be displaced by a Hamiltonian isotopy; the reason is that its Floer homology complex is non-vanishing and invariant under such deformations \cite{Floer}.

On the contrary, the results in this paper concern the construction of Lagrangians on the flexible side of symplectic topology, where the construction of the obstructions coming from pseudoholomorphic curves and Floer theory are known to fail, and where many phenomena have been shown to be governed by h-principles. More specifically, we will use the flexibility manifested in contact manifolds that are overtwisted in the sense of Eliashberg \cite{Eliashberg:Overtwisted} in dimension three in order to deduce the existence of closed exact Lagrangians in their symplectisations. The symplectisation of a contact manifold is an open symplectic manifold and, further, due to its concave cylindrical end, it does not have bounded geometry at infinity in the sense of \cite{Gromov}. Certain symplectisations can still be embedded in symplectic manifolds that have bounded geometry at infinity, but this is not the case for symplectisation of overtwisted contact manifolds; see e.g.~work \cite{Eliashberg:Filling} by Eliashberg for closed contact manifolds and more recent work by Cant \cite{Cant} for open contact manifolds. There is hence a priori no reason to expect that Floer homology is well-behaved in these symplectic manifolds. Indeed, since closed exact Lagrangians in the symplectisation are automatically displaceable by Hamiltonian isotopies (the Liouville flow restricted to exact Lagrangians can be extended to a global Hamiltonian isotopy); the existence of such Lagrangians, as established by Theorem \ref{thm:main}, implies that Floer homology, if it is well-defined, must be acyclic in this geometric setting.

 Even though the technical details have not been fully established, one should be able to construct Floer homology of exact Lagrangian submanifolds in symplectisations with coefficients in the contact homology DGA generated by periodic Reeb orbits. The invariance property of this version of Floer theory, once established, should imply that the contact homology DGA is necessarily acyclic whenever the symplectisation contains a closed exact Lagrangian. Note that Avdek showed in \cite{Avdek} that an acyclic contact homology DGA does not imply overtwistedness for three-dimensional contact manifolds (the converse, however, is true).

In this paper we produce new examples of closed exact Lagrangians in symplectisations, and also use the same ideas to prove that non-interlinkedness holds in the dynamical sense defined by \cite{EntovPolterovich} by using a similar trick.

\subsection{First result: Existence of closed exact Lagrangians} The existence of exact Lagrangians inside symplectisations has previously been established in high dimension. The first example was due to Muller in \cite{Muller}, who constructed an exact Lagrangian 3-sphere in the symplectisation of a five-dimensional contact manifold. Later, Murphy \cite{Murphy:Closed} used the h-principle for Lagrangian caps by Eliashberg--Murphy \cite{EliashbergMurphy} in order to produce plenty of examples of closed exact Lagrangian inside symplectisations of any high-dimensional overtwisted contact manifold; roughly speaking, they exist as soon as the formal homotopy-theoretic obstruction vanishes. The construction in the latter article is only applicable in symplectisations of dimension at least six. The reason is that, for their construction of a Lagrangian cap, one needs the input of a formally embedded Lagrangian cap which is cylindrical over a loose Legendrians. The construction then makes heavy use of the h-principle for loose Legendrians from \cite{LooseLeg}, which exists only in contact manifolds of dimension at least five. In this article, we show that closed exact Lagrangians exist also in four-dimensional symplectisations of overtwisted contact three-manifolds. In particular, this answers \cite[Question 1]{Cant} in recent work by Cant.

Below in Lemma \ref{lma:immersion} we will define an embedded non-exact Lagrangian torus $L_{\Lambda,\epsilon S^1} \subset S(Y,\alpha)$ that satisfies the following property: $L_{\Lambda,\epsilon S^1} \cap \{\tau=0\}$ is the Legendrian two-copy link that consists of $\Lambda$ and its time-$\epsilon$ Reeb push-off, $L_{\Lambda,\epsilon S^1} \cap \{\tau \le 0\}$ is an exact cylinder filling of the two copy link, while $L_{\Lambda,\epsilon S^1} \cap \{\tau\ge 0\}$ is an exact cylinder cap of the two-copy link. (The non-exactness of the entire torus comes from the fact that the symplectic action has a potential difference at the two Legendrian boundary components.) The following theorem gives a condition under which this torus is smoothly isotopic through Lagrangians to an exact Lagrangian torus.
 
\begin{thm}
\label{thm:main}
For any overtwisted contact three-manifold $(Y^3,\alpha)$, there exist exact Lagrangian embeddings of tori and Klein bottles in its symplectisation $S(Y,\alpha)$. 

Furthermore, for $\epsilon>0$ sufficiently small and where $\Lambda \subset (Y,\alpha)$ is a Legendrian knot with overtwisted complement.Then: 
\begin{itemize}
\item $L_{\Lambda,\epsilon S^1} \subset S(Y,\alpha)$ is Lagrangian isotopic to an exact Lagrangian torus by an isotopy supported in the subset $\{\tau \ge 0\} \subset S(Y,\alpha)$.
\item When the first Chern class of $(Y, \ker \alpha)$ vanishes and the Legendrian knot $\Lambda$, moreover, is null-homologous and has vanishing rotation number ${\tt rot}(\Lambda)=0$, there exist exact Lagrangian tori whose Maslov class vanish, that can be taken to coincide with $L_{\Lambda,\epsilon S^1} \subset S(Y,\alpha)$ in the subset $\{\tau \le 0\}$.
\end{itemize}
\end{thm}
\begin{rmk}
The technique can be applied to also produce examples in symplectisations of higher dimensional contact manifolds that are diffeomorphic to e.g.~$S^1 \times M$. Such examples are already known to exist from the work of Murphy \cite{Murphy:Closed}, but our construction is based upon more elementary and explicit techniques.
\end{rmk}

In view of the above result, we point out the following natural open questions.
\begin{quest}
Are there tight (necessarily non-fillable) contact manifolds whose symplectisations contain closed exact Lagrangians?
\end{quest}

Further, we note that the Lagrangians tori produced by Theorem \ref{thm:main} all bound smoothly embedded solid tori, and are thus unknotted.
\begin{quest}
 Does there exist knotted examples of exact Lagrangian tori in symplectisations of three-dimensional overtwisted contact manifolds?
\end{quest}
We expect that the flexibility properties of Legendrian knots in overtwisted contact manifolds allows us to deform the torus by replacing a trivial Lagrangian cylinder by the Lagrangian trace induced by a non-trivial loop of Legendrian embeddings, so that the resulting torus becomes smoothly knotted.

\subsection{Second Result: Dynamical non-interlinkedness of topologically linked Legendrians}

Entov--Polterovich proved the following ``dynamical interlinkedness'' result in \cite[Theorem 1.5]{EntovPolterovich}. Let $\Lambda \subset J^1M \setminus j^10$ be a closed Legendrian for which there exists a unique transverse Reeb chord from $j^10$ to $\Lambda$, and such that $\Lambda$ has a Chekanov--Eliashberg algebra that admits an augmentation. Then $j^10$ and $\Lambda$ are {\bf interlinked} in the following sense: for any contact isotopy $\phi^t$ whose (possibly non-autonomous) generating Hamiltonian is uniformly positive $H_t \ge c$ for some constant $c>0$, an intersection of $\phi^t(j^10)$ and $\Lambda$ appears in finite time depending on the constant $c$ and the length of the aforementioned Reeb chord. Here we give a simple construction of a stabilised Legendrian $\Lambda$ (its Chekanov--Eliashberg algebra does hence not admit an augmentation), with a unique chord from $j^10$ to $\Lambda$, but for which the conclusion of the aforementioned interlinkedness result fails.

 The condition on the existence of a transverse Reeb chord means that the two Legendrians $j^10$ and $\Lambda$ are non-trivially linked as a topological submanifolds. One could believe that the contact isotopies of $j^10$ that are generated by positive contact Hamiltonians eventually must perform a topological unlinking of the two Legendrians, which then would be a mechanism that forces the appearance of intersection points. However, as is shown by the following result, this is not the case.

Consider the Legendrian $j^1\cos{\theta} \subset J^1S^1$. There is a unique transverse Reeb chord from $j^10$ to $j^1\cos{\theta}$ given by $\{\theta=p=0, \; z \in [0,1]\}$. Denote by $\Lambda$ a small stabilisation of $j^1\cos{\theta}$, where the stabilisation is small and performed away from this Reeb chord. Consequently, we can ensure that $\Lambda$ also satisfies the property that there exists a unique transverse Reeb chord from $j^10$ to $\Lambda$.

\begin{thm}
\label{thm:noninterlinkedness}
Let $\Lambda \subset J^1S^1$ be the stabilised Legendrian constructed above, which thus has the property that there is a unique Reeb chord from $j^10$ to $\Lambda$, which moreover is transverse. There exists a uniformly positive autonomous Hamiltonian
$$H \colon J^1S^1 \to [c,+\infty), \:\: c >0,$$
such that the corresponding contact isotopy $\phi^t \in \OP{Cont}_0(J^1S^1)$ satisfies
$$\phi^t(j^10) \cap \Lambda = \emptyset$$
for all $t\in \R$. We may assume that $H =1$ holds outside of some compact subset.
\end{thm}
We conclude the following consequences for the dynamics of the Reeb flow.
\begin{cor}
There exists a contact form
$$\alpha_1=e^f\alpha_0=e^f(dz-p\,d\theta)$$
on $J^1S^1$ for some smooth $f \colon Y \to \R$, which is standard outside of a compact subset, and for which there are no Reeb chords from $j^10$ to $\Lambda$.

In addition, for any generic path $\alpha_t=e^{f_t}\alpha_0$ of contact forms that connects $\alpha_0$ to $\alpha_1$ through contact forms that agree outside of some fixed compact subset, there must exist infinitely many Reeb chords from $j^10$ to $\Lambda$ for some time $t \in (0,1)$.
\end{cor}
\begin{proof}
The contact form $\alpha_1$ exists by Lemma \ref{lma:reeb} applied to the positive contact isotopy produced by Theorem \ref{thm:noninterlinkedness}.

The existence of infinitely many Reeb chords is then proven as follows.

The derivative of the length $\ell(c_t)$ of a smooth family $c_t(s)$ of Reeb chords for $\alpha_t$ with starting point on $j^10$ and endpoint on $\Lambda$ is given by
\begin{equation}
\label{eq:zero}
\frac{d}{dt}\ell(c_t)= \frac{d}{dt}\int_0^t\int_0^1 d_{\R_t \times Y}\boldsymbol{\alpha}\left(\partial_t c, \partial_s c \right)ds\,d\tau= \int_{c_t} \iota_{\partial_t c}d_{ Y} \alpha_t+\dot{f}_t \alpha_{ t}=\int_{c_t} \dot{f}_t \alpha_{ t}
\end{equation}
and hence satisfies $|\frac{d}{dt}\ell(c_t)| \le \max |\partial_t {f}_t| |\ell(c_t)|$. Here $\boldsymbol{\alpha}$ denotes the one-form on $\R_t \times Y$ that is induced by the smooth path of one-forms $\alpha_t \in \Omega^1(Y)$; this means that $d_{\R_t \times Y}\boldsymbol{\alpha}=d_Y\alpha_t+\dot{f}_t dt\wedge \alpha_t$. For the first equality in Equation \eqref{eq:zero} we have used the fact that $c(t,s)=(t,c_t(s)) \in \R \times Y$ is a curve for which $\boldsymbol{\alpha}(\partial_t c)$ vanishes along the endpoints $s=0,1.$ The latter property holds since the endpoints of the curves $c_t$ live on Legendrian submanifolds.

Assuming that there are only finitely many Reeb chords from $j^10$ to $\Lambda$ for any $t \in (0,1)$ in a generic path of contact forms, we conclude that there is a uniform bound on their lengths. We can then use a standard bifurcation argument to show that there must be an odd number of Reeb chords from $j^10$ to $\Lambda$ for all generic moments in time. (However, if there are infinitely many chords, then an infinite chain of births and deaths can occur.)
\end{proof}

\subsection{Main trick used in the proofs}

The main mechanism used in the proofs of the two theorems above is the following elementary result that was first observed by Colin--Ferrand--Honda in \cite[Theorem 3.i]{Colin:Positive}: a stabilised Legendrian knot sits in a loop that is generated by a positive contact isotopy. This positive contact isotopy can moreover be assumed to be autonomous; see Lemma \ref{lma:posloop} for the precise statement that we need. In connection with the interlinkedness studied in \cite{EntovPolterovich}, we note that a pair of Legendrians for which one admits a positive autonomous loop in the complement of the other is not interlinked.

For the construction of the exact Lagrangian embeddings we need a {\bf two-copy Legendrian} $\Lambda \sqcup \phi^\epsilon_{R_\alpha}(\Lambda)$, i.e.~a Legendrian link that consists of $\Lambda$ and a small push-off by the Reeb flow, that satisfies the additional property that one of the components admits a positive loop in the complement of the other. Even though certain Legendrians are known to be contained in a positive loop (e.g. any closed Legendrian in the contact vector space; see \cite{Dimitroglou:Positive}), there are examples when the loop cannot be constructed in the complement of a second Legendrian that links the first one non-trivially.

More precisely, in work \cite{Dimitroglou:Positive} by the author joint with Colin--Chantraine, it was shown that a component of a two-copy does not admit a positive loop in the complement of the other component when the ambient contact manifold is e.g.~the standard tight contact sphere, and when the Chekanov--Eliashberg algebra of the Legendrian admits an augmentation. This result applies, in particular, to the two-copy of the standard Legendrian unknot in the standard contact vector space or in the tight contact sphere. In the proof of Theorem \ref{thm:main} the overtwistedness is crucial for the following reason:when there is an overtwisted disc in the complement of the two-copy Legendrian, the two-copy link is stabilised. (This means in particular that each component is stabilised in the complement of the other component.) We can then use Lemma \ref{lma:posloop} to construct an autonomous positive loop of one component in the complement of the other.

The overtwistedness is indeed necessary for the strategy of the proof of Theorem \ref{thm:main}. In fact, the property that some two-copy Legendrian to be stabilised can be used as a characterisation of overtwistedness.
\begin{thm}
\label{thm:twocopy}
 Let $\Lambda \sqcup \phi^\epsilon_{R_\alpha}(\Lambda) \subset (Y^{2n+1},\xi)$ be a Legendrian link that consists of a sphere together with its small Reeb push-off. If the link is stabilised, then the contact manifold $(Y,\xi)$ is overtwisted. 
\end{thm}
\begin{proof}
 
Start by performing a contact $-1$-surgery along one of the two spherical components $\Lambda \sqcup \phi^\epsilon_{R_\alpha}(\Lambda)$ of the link; denote this component by $\Lambda'$, and denote the resulting contact manifold by $Y_{\Lambda'}^-.$ We pick $\Lambda'$ so that the remaining component $\Lambda \sqcup \phi^\epsilon_{R_\alpha}(\Lambda) \setminus \Lambda'$ is stabilised inside the complement $Y \setminus \Lambda'$. This produces a contact manifold $Y_{\Lambda'}^-$ in which the Legendrian belt sphere $\Lambda'' \subset Y_{\Lambda'}^-$ produced by the surgery is stabilised. To that end, note that $\phi^{\pm \epsilon}_{R_\alpha}(\Lambda') \subset Y \setminus \Lambda'$ both are Legendrian isotopic to the belt sphere $\Lambda''$ when considered inside $Y_{\Lambda'}^-$.

The result then follows immediately from the standard fact that contact $+1$-surgery performed along the belt sphere $\Lambda'' \subset Y_{\Lambda'}^-$ produced by the $-1$-surgery gives back the original contact manifold $(Y_{\Lambda'}^-)_{\Lambda''}^+=Y$, in combination with the result that contact $+1$-surgery performed on a stabilised Legendrian sphere produces an overtwisted contact manifold. We refer to \cite{Ozbagci} for the 3-dimensional case and to \cite[Part (5) $\Rightarrow$ (1) of Theorem 1.1]{Casals:Geometric} for the higher-dimensional case of the latter result. This finishes the proof.

For completeness, we end by recalling how overtwistedness is established. Overtwistedness of $Y$ is equivalent to the property that the standard unknotted Legendrian sphere $\Lambda_0 \subset Y$ admits a stabilisation; this follows by Giroux' criterion in dimension three \cite{Giroux:Convexite}, or by \cite{Casals:Geometric} in higher dimensions. To show that the standard Legendrian sphere $\Lambda_0$ is stabilised if it lives inside the $+1$-surgery $Z_{\Lambda''}^+$ performed on a stabilised Legendrian sphere $\Lambda'' \subset Z$, one can argue as follows. Consider a standard Legendrian neighbourhood of $\Lambda'' \subset Z$ contactomorphic to $U^*\R^n$, in which $\Lambda''$ is identified with the fibre $U^*_0\R^n$ over the origin. It is sufficient to show that, if $\Lambda''$ is stabilised in $Z$, then it is also stabilised in the complement of the Legendrian push-off given by $U^*_{\mathbf{x}}\R^n \subset U^*\R^n$ for some small $\mathbf{x} \in \R^n \setminus \{0\}$. Indeed, if this is the case, then one can show that $\Lambda_0 \subset Z^+_{\Lambda''}$ is stabilised. Namely, the $+1$-surgery along the push-off of $\Lambda'' \subset Z$ identified with $U^*_{\mathbf{x}}\R^n$ produces the same contact manifold as for the original $+1$-contact surgery along $\Lambda''$. However, in the latter contact manifold, one immediately identifies $U^*_{0}\R^n$ with the standard Legendrian sphere $\Lambda_0\subset Z_{\Lambda''}^+$, which thus is stabilised as sought.
\end{proof}

The above result raises the following natural question:
\begin{quest}
 Are there two-copy links in non-overtwisted contact manifolds where one of the components admits a positive loop in the complement of the other?
 \end{quest}

\section{Prerequisites}

\subsection{Contact manifolds and symplectisations}
A {\bf contact manifold} is an odd-dimensional manifold $(Y^{2n+1},\xi)$ equipped with a maximally non-integrable field of tangent hyperplanes $\xi \subset TY$. We will only consider the co-oriented case, i.e.~when $\xi=\ker \alpha$ for a so-called contact one-form $\alpha \in \Omega^1(Y)$; the contact condition translates to the requirement that $d\alpha|_{\ker \alpha}$ is a non-degenerate 2-form. An $n$-dimensional submanifold $\Lambda^n \subset (Y^{2n+1},\xi)$ is {\bf Legendrian} if $T\Lambda \subset \xi$.

\begin{exm}
The jet-space $(J^1M=T^*M \times \RR_z,\alpha_0)$ of a smooth manifold is an important example of a contact manifold, where $\alpha_0 \coloneqq dz-p\,dq$ and $p\,dq \in \Omega^1(T^*M)$ denotes the tautological one-form. Here we will only need the cases $M=S^1$ and $M=\RR$.
\begin{itemize}
\item $J^1S^1=T^*S^1\times\RR_z=S^1_\theta \times \RR_p\times\RR_z$ where the tautological one-form on $T^*S^1$ takes the simple form $p\,dq=p\,d\theta$.
\item $\left(J^1\RR=\RR^3_{q,p,z},\alpha_0=dz-p\,dq\right)$ is the standard contact Darboux ball.
\end{itemize}
\end{exm}
\begin{exm}
The one-jet $j^1f$ of a smooth function $f \colon M \to \RR$ is a Legendrian submanifold. A Legendrian submanifold in $J^1M$ can be recovered from its so-called {\bf front projection}, which is the canonical projection $J^1M \to M \times \RR_z$. See Figure \ref{fig:unknot} for a Legendrian in $J^1\R$ that is not a section.
\end{exm}

Given a choice of a contact one-form $\alpha \in \Omega^1(Y)$, the {\bf symplectisation} of $(Y,\alpha)$ is symplectic $2(n+1)$-dimensional manifold
$$ S(Y,\alpha) \coloneqq \left(\RR_\tau \times Y,d(e^\tau\alpha)\right)$$
equipped with the natural choice of primitive $e^\tau\alpha$ of its symplectic form. Two different contact forms for the same contact structure differ by a factor $e^g$ for some smooth function $g \colon Y \to \R$. The symplectisations for two different choices of contact forms $\alpha$ and $e^g\alpha$ are related by a symplectomorphism of the symplectisations
\begin{gather*}
S(Y,\alpha) \to S(Y,e^g\alpha),\\
(\tau,y) \mapsto (\tau-g(y),y),
\end{gather*}
that preserves the primitives of the symplectic forms; such a map is called an {\bf exact symplectomorphism}.

A half-dimensional submanifold $L \subset S(Y,\alpha)$ is {\bf Lagrangian} if the symplectic form $d(e^\tau\alpha)$ vanishes on $TL$. It is {\bf exact Lagrangian} if $e^\tau\alpha|_{TL}$ moreover is an exact one-form.

\subsection{Contact Hamiltonians}

For a contact manifold $(Y,\alpha)$ endowed with a contact form $\alpha$, there is a bijective correspondence between contact isotopies
$$\phi^t \colon (Y,\ker \alpha) \to (Y,\ker \alpha),$$
i.e.~smooth isotopies that preserve the contact distribution $\ker \alpha \subset TY$, and so-called {\bf contact Hamiltonians}, which are smooth $t$-dependent functions $H_t \colon Y \to \R$. A contact Hamiltonian that does not depend on $t$ is called {\bf autonomous}. Given a contact isotopy $\phi^t$ with infinitesimal generator $\dot{\phi^t} \circ (\phi^t)^{-1} \in \Gamma(TY)$, the corresponding contact Hamiltonian is given by
$$ H_t = \alpha \left(\dot{\phi^t}\circ (\phi^t)^{-1}\right) \colon Y \to \R.$$
Conversely, any such function $H_t$ gives rise to a contact isotopy via the ODE
$$ \begin{cases}
\alpha\left(\dot{\phi^t} \circ (\phi^t)^{-1} \right)=H_t,\\
\left.\iota_{\dot{\phi^t} \circ (\phi^t)^{-1}}d\alpha\right|_{\ker \alpha}=\left.-dH_t\right|_{\ker\alpha}.
\end{cases}$$
This correspondence can also be expressed in the following manner by referring to the symplectisation. A contact isotopy $\phi^t$ lifts to a symplectic isotopy
\begin{gather*}
S(Y,\alpha) \to S(Y,\alpha),\\
(\tau,y) \mapsto (\tau-f_t(y),\phi^t(y)),
\end{gather*}
where function $f_t(y)$ is called the {\bf conformal factor} and is determined by $(\phi^t)^*\alpha=e^{f_t}\alpha$. This symplectic isotopy is in fact generated by a Hamiltonian of the form $e^\tau H_t$, where $H_t \colon Y \to \R$ is the corresponding contact Hamiltonian. 
\begin{exm}
The {\bf Reeb flow} induced by the contact form $\alpha$ is the contact isotopy
$$ \phi^t_{R_{\alpha}} \colon (Y,\alpha) \to (Y,\alpha) $$
generated by the autonomous contact Hamiltonian $H \equiv 1$. The Reeb flow can be equivalently characterised as the flow induced by the vector field $R_\alpha$ uniquely determined by
$$ \alpha(R_\alpha) \equiv 1, \:\: \iota_{R_\alpha}d\alpha=0.$$
Note that $\phi^t$ preserves the contact form $\alpha$.
\end{exm}
The Reeb flows for different contact forms on $(Y,\ker\alpha)$ correspond precisely to the contact isotopies that are generated by autonomous positive Hamiltonians by the following basic result.
\begin{lma}
\label{lma:reeb}
The strictly positive autonomous contact Hamiltonian $H \colon Y \to \R$ for the contact form $\alpha$ generates the Reeb flow for the contact form $e^{-H}\alpha$.
\end{lma}
\begin{proof}
The Hamiltonian $e^\tau \cdot H$ has level sets $\{e^\tau \cdot H = e^c\}=\{\tau=c-\log{H}\}$. Consider the exact symplectomorphism
\begin{gather*}
\Phi \colon S(Y,\alpha) \to S(Y,H^{-1}\alpha),\\
(\tau,y) \mapsto (\tau+\log{H},y),
\end{gather*}
under which $e^\tau \cdot H \circ \Phi^{-1} \equiv e^\tau$. In other words, the Reeb flow for the contact form $e^{-H}\alpha$, which for the same contact form is generated by a contact Hamiltonian constantly equal to one, is generated by the contact Hamiltonian $H$ under the correspondence induced by $\alpha$.
\end{proof}

Here we will mostly be concerned with the behaviour of the contact isotopy along a Legendrian submanifold; this is a so-called {\bf Legendrian isotopy}. It is a standard fact that any smooth isotopy $\Lambda_t \subset (Y,\alpha)$ through Legendrian submanifolds can be induced by an ambient contact isotopy, which thus is generated by a global contact Hamiltonian \cite{Geiges}. Note that the value of the contact Hamiltonian along the image of the Legendrian is uniquely determined by the path of submanifolds (i.e.~the unparametrised isotopy).

\subsection{Constructions of Lagrangian cobordisms}
Here we consider different methods for constructing Lagrangian submanifolds and cobordism inside the symplectisation.

First, recall that a Lagrangian $L \subset S(Y,\alpha)$ is {\bf cylindrical (over $\Lambda$)} in a subset of the form
$$U \times Y \subset S(Y,\alpha), \:\: U \subset \R,$$
if the identity
$$ L \cap (U \times Y) = U\times \Lambda \subset S(Y,\alpha)$$
is satisfied. It follows by the Lagrangian property that $\alpha$ vanishes on $\Lambda \subset Y$, hence the latter is a Legendrian submanifold.

Given a Legendrian isotopy $\Lambda_t \subset (Y,\alpha)$, $t \in [0,1]$, one can construct a properly embedded Lagrangian cobordism $L \subset S(Y,\alpha)$ diffeomorphic to a cylinder, which is cylindrical over the Legendrian $\Lambda_0$ (resp. $\Lambda_1$) in the subset $(-\infty,-T]\times Y$ (resp. $[T,+\infty) \times Y$) for some $T \gg 0$. There are several constructions of this so-called {\bf Lagrangian trace cobordism} in the literature; see e.g.~\cite[Theorem 1.2]{Chantraine:Concordance} by Chantraine.

Next we will also need the following technique for constructing Lagrangian cobordisms from the Reeb flow; such a construction has been used previously e.g.~by Mohnke in \cite{Mohnke}.
\begin{lma}
\label{lma:immersion}
If $\Lambda \subset (Y,\alpha)$ is an embedded Legendrian, then for any smooth immersion
$$\gamma=(\gamma_1,\gamma_2) \colon [0,1] \looparrowright \R^2,$$
the induced immersion
\begin{gather*}
L_{\Lambda,\gamma} \colon [0,1] \times \Lambda \looparrowright S(Y,\alpha),\\
(t,x) \mapsto \left(\gamma_1(t),\phi^{\gamma_2(t)}_{R_\alpha} (x)\right),\:\: t \in I, x \in \Lambda,
\end{gather*}
is Lagrangian. The images $L(t_i,x_i)$, $i=1,2$, coincide if and only if $\gamma_1(t_1)=\gamma_1(t_2)$, and $x_1,x_2 \in \Lambda$ are the endpoints of a Reeb chord in $(Y,\alpha)$ of length $|\gamma_2(t_1)-\gamma_2(t_2)|$.

Furthermore, the symplectic action form $L_{\Lambda,\gamma}^*e^\tau\alpha=dh$ is exact, where the primitive $h$ can be taken to be constantly equal to $h\equiv 0$ and $h \equiv \int_{[0,1]} e^{\gamma_1} d\gamma_2$ on the boundary component $\{t=0\}$ and $\{t=1\}$, respectively.
\end{lma}

Note that the above lemma can be used also for constructing Lagrangian immersions of $S^1 \times \Lambda$; it suffices to consider curves $\gamma$ that have images that are smooth immersions of closed curves.

\subsection{Legendrians in complements of overtwisted discs}
\label{sec:destabilising}

A three-dimensional contact manifold is called {\bf overtwisted} if it contains an overtwisted disc; see \cite{Eliashberg:Overtwisted}. The standard model of the overtwisted disc can be described as follows. Consider the standard overtwisted contact structure on the three-dimensional vector space
$$\left(\RR^2_{r,\theta}\times \RR_z,\alpha_{ot}=\cos{r}\,dz+r\sin{r}\,d\theta\right),$$
where $(r,\theta)$ denotes polar coordinates on $\RR^2$. Note that this contact form is invariant under translations of the $z$-coordinate. The standard overtwisted disc is the germ of the contact structure defined on the disc
$$D^2_{ot} \coloneqq \{r \le 2\pi, \: z=0\} \subset \RR^3.$$
Note that the boundary
$$\Lambda_0 \coloneqq \partial D^2_{ot}$$
is Legendrian, and that $\partial_z$ is a contact vector field transverse to $D^2_{ot}$ that coincides with the Reeb vector-field for $\alpha_{ot}$ along $\Lambda_0$.

It is a well-known fact that any Legendrian contained inside the complement of an overtwisted disc is Legendrian isotopic to a stabilised Legendrian; see work by Eliashberg--Fraser \cite{EliashbergFraser} as well as Dymara \cite{Dymara} for an even stronger statement. We present this result by using the construction of \emph{cusp connected sum}, that was introduced in \cite{OnConnectedSum} by Etnyre--Honda. More precisely, we here use the formulation from \cite{Dimitroglou:Ambient}, where the cusp connected sum of two Legendrians knots is constructed from the data of the choice of an embedded Legendrian arc $\eta$ with boundary on the two Legendrians, where the tangents at the two boundary points are assumed to be normal to the Legendrians. Given two disjoint Legendrians $\Lambda_i$, $i=0,1$, and such an embedded Legendrian arc $\eta$, we obtain $\Lambda_0 \,\sharp_\eta\, \Lambda_1$ which is constructed in an arbitrarily small neighbourhood of $\Lambda_0 \cup \Lambda_1 \cup \eta$; see Figure \ref{fig:cuspconnectsum} for a description in terms of a locally defined front projection.

Recall the following standard result for the cusp-connected sum with the standard Legendrian unknot $\Lambda_{std}$.
\begin{lma}
\label{lma:conn}
 Consider a standard Legendrian unknot $\Lambda_{std}$ that is contained in a Darboux ball $B$ that is disjoint from a second Legendrian $\Lambda$, and in which $\Lambda_{std} \cup \eta$ is contactomorphic to the configuration shown in Figure \ref{fig:unknot}. Then, the cusp connected sum $\Lambda \,\sharp_\eta\, \Lambda_{std}$ is Legendrian isotopic to $\Lambda$ by a Legendrian isotopy that can be assumed to have support in an arbitrarily small neighbourhood of $\eta \cup B$.
\end{lma}

We also recall that any Legendrian knot $\Lambda$ has canonically defined (positive and negative) {\bf stabilisations} $\Lambda^{stab}$ constructed in an arbitrarily small standard contact jet-neighbourhood $J^1\Lambda \supset \Lambda$. The local modifications can be described by the front projection to the right in Figure \ref{fig:stab} or by its reflection under $z\mapsto -z$, depending on the sign of the stabilisation. (We will not need to care about the sign of the stabilisation here.)

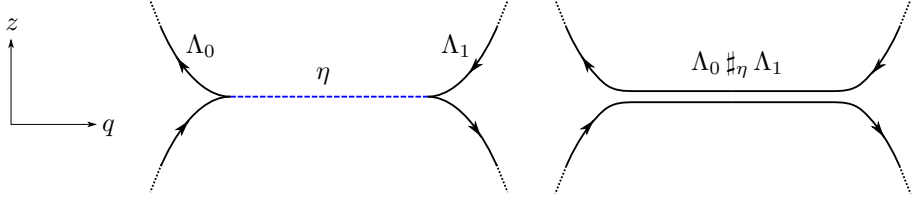
\begin{figure}[htp]
\centering
\vspace{3mm}
\labellist
	\pinlabel $z$ at 1 46
	\pinlabel $\Lambda_0$ at 52 40
	\pinlabel $\Lambda_1$ at 121 40
\pinlabel $\Lambda_0\,\sharp_\eta\,\Lambda_1$ at 197 35
 	\pinlabel $\eta$ at 85 32
\pinlabel $q$ at 27 18
	\endlabellist
\includegraphics[width=12cm]{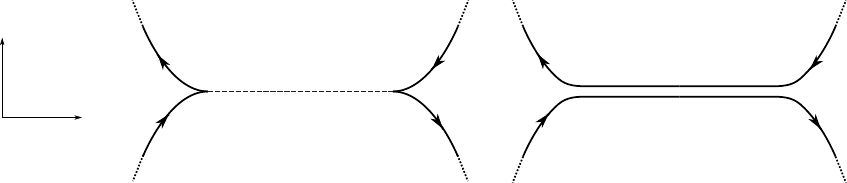}
\caption{Left: The front projection of the local model of a Legendrian arc $\eta$ (shown as a dashed blue line) with boundary on a pair of Legendrians $\Lambda_0 \sqcup \Lambda_1$. Right: the front projection of the result of a cusp connected sum $\Lambda_0 \,\sharp\, \Lambda_1$.}
\label{fig:cuspconnectsum}
\end{figure}

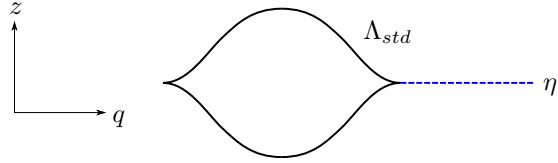
\begin{figure}[htp]
\centering
\vspace{3mm}
\labellist
\pinlabel $z$ at 1 38
\pinlabel $q$ at 27 10
\pinlabel $\Lambda_{std}$ at 95 32
\pinlabel $\eta$ at 136 18
	\endlabellist
\includegraphics[height=2cm]{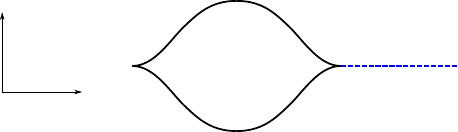}
\caption{The front projection of the standard Legendrian unknot contained inside a Darboux ball $(\R^3_{q,p,z},dz-p\,dq)$. The dashed Legendrian arc $\eta$ can be used for performing a connected sum $\Lambda \,\sharp_\eta\,\Lambda_{std}$ that preserves the Legendrian isotopy class of $\Lambda$.}
\label{fig:unknot}

\end{figure}

\begin{lma}[Theorem 2.1 in \cite{Murphy:Plastik}]
 \label{lma:stabilisedunknot}
 \begin{enumerate}
 \item Let $\Lambda_{0} \subset (Y^3,\alpha)$ be the boundary of an overtwisted disc. The standard Legendrian unknot contained in a Darboux ball is contact isotopic to the Legendrian $\Lambda_0^{stab}$ obtained from $\Lambda_0$ by one stabilisation.
 \item Any Legendrian $\Lambda$ that is contained in the complement of an overtwisted disc is Legendrian isotopic to a cusp connected sum $\Lambda \,\sharp_\eta\, \Lambda_0^{stab}$ for a suitable choice of Legendrian arc $\eta$ (see Lemma \ref{lma:conn}). The latter, in turn, lives in the same Legendrian isotopy classes as the stabilisations $(\Lambda \,\sharp_\eta \, \Lambda_0)^{stab}$ and $\Lambda^{stab} \,\sharp_\eta \, \Lambda_0$
 \end{enumerate}
\end{lma}
\begin{proof}
(1): The standard overtwisted disc $D^2_{ot}$ is convex in the sense of Giroux \cite{Giroux:Convexite} and has a single closed dividing curve. The stabilisation of $\Lambda_0=D^2_{ot}$ can thus be realised as (a smoothed version of) a semicircle in the boundary of $D^2_{ot}$ union a diameter . This is the standard unknot of ${\tt tb}=-1$, and the half-disc is contained inside a Darboux ball.

(2): It suffices to find an arc $\eta$ as in Lemma \ref{lma:conn} for which $\eta \cup \Lambda_0^{stab}$ still contains a stabilisation (i.e.~so that the latter union of Legendrians can be destabilised). This can be done by making sure that the arc $\eta$ only touches $D^2_{ot}$ at a single boundary point, i.e.~at its endpoint located in the standard Legendrian unknot.
\end{proof}

A crucial tool in this paper is the following existence of positive loops of stabilised Legendrians. The first version of existence of such a positive loop was shown by Colin--Ferrand--Pushkar in \cite[Theorem 3.i]{Colin:Positive}; in addition see \cite[Example 1.3]{Dimitroglou:Positive}. Similar results have more recently been proven for arbitrarily loose Legendrians in higher dimension by Liu \cite{Liu:Positive}.

\begin{lma}
\label{lma:posloop}
The Legendrian isotopy class obtained by stabilisation of the zero-section $j^10 \subset (J^1S^1,\ker\alpha_0)$ has a representative $(j^10)^{stab}$ that is contained inside a subset of the form
$$\left\{p \in (0,\epsilon), \: z \in (-\epsilon,\epsilon) \right\} \subset J^1S^1,$$
where $\epsilon>0$ can be taken to be arbitrarily small.

The contact isotopy $\phi^t(\theta,p,z)=(\theta-t,p,z)$ satisfies $\phi^{2\pi}=\phi^0=\id$ and is generated by the autonomous contact Hamiltonian $H = p$ which is positive in the subset $\{ p > 0\}$. In particular, it induces a positive loop of the aforementioned Legendrian $(j^10)^{stab}$.
\end{lma}
\begin{proof}
As described in \cite[Example 1.3]{Dimitroglou:Positive}, the once stabilised zero section admits a representative where the front has a zig-zag; see Figure \ref{fig:stab}. One can make sure that the slope of the front is everywhere contained inside $(0,a)$ after a Legendrian isotopy. This means that the Legendrian is contained inside $\{p \in (0,a)\}$. It finally suffices to apply the contactomorphism $((\theta,p,z) \mapsto (\theta,e^{-t}p,e^{-t}z)$ for $t \gg 0$ to move the Legendrian into the sought subset.
\end{proof}

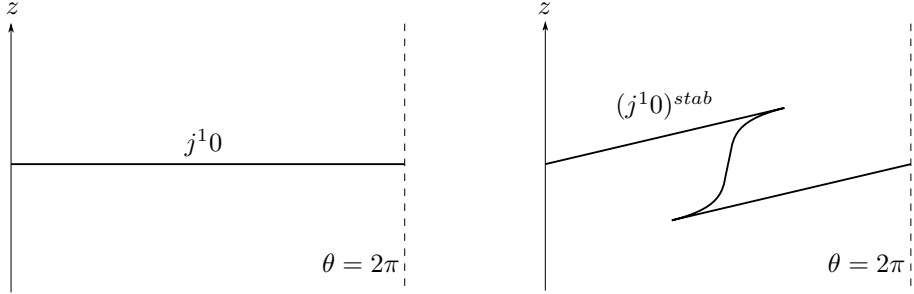
\begin{figure}[htp]
\centering
\vspace{3mm}
\labellist
	\pinlabel $z$ at 1 76
	\pinlabel $j^10$ at 52 40
 	\pinlabel $(j^10)^{stab}$ at 175 50
 	\pinlabel $z$ at 143 76
\pinlabel $\theta=2\pi$ at 94 8
\pinlabel $\theta=2\pi$ at 229 8
	\endlabellist
\includegraphics[width=12cm]{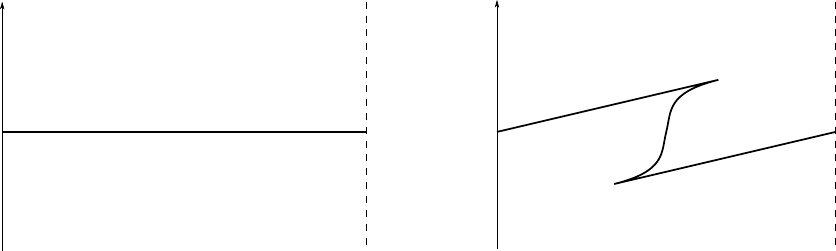}
\caption{Left: The front projection of the zero section $j^10 \subset J^1S^1$. Right: The front projection of a representative of the stabilisation of $(j^10)^{stab}\subset J^1S^1$ of the zero-section contained in the subset $p>0$. Analogous representatives exist also for the stabilisation of the opposite sign.}
\label{fig:stab}

\end{figure}

Finally we also need the following result.
\begin{lma}
\label{lma:reversing}
Let $\Lambda_{0} \subset (Y^3,\alpha)$ be the boundary of an overtwisted disc in a contact manifold. There exists a contact isotopy that reverses the orientation of $\Lambda_0$. Moreover, there exists such an isotopy that is contained inside an arbitrarily small neighbourhood of the overtwisted disc.
\end{lma}
\begin{proof}
We will prove the statement in the case when $Y$ is an arbitrarily small neighbourhood of the overtwisted disc.

Any overtwisted disc can be pushed off from itself by a contact isotopy whose infinitesimal generator is transverse to the disc; one can simply take the translation of the $z$-coordinate in the standard model of the overtwisted disc described above.

Consider the boundary $\Lambda_0^+$ of a second overtwisted disc that is obtained by the above displacement by a contact isotopy. Let $\eta$ be a Legendrian arc that connects $\Lambda_0$ to $\Lambda_0^+$. We orient $\Lambda_0$ and $\Lambda_0^+$ so that the orientations are compatible with the operation of cusp connected sum $\Lambda_0 \,\sharp_{\eta}\, \Lambda_0^+$ as shown in Figure \ref{fig:cuspconnectsum}.

We may choose a single arc $\eta$ in neighbourhoods of the two disjoint overtwisted discs, so that Lemma \ref{lma:conn} can be applied to give a Legendrian isotopy
\begin{equation}
\label{eq:one}
\Lambda_0 \,\sharp_{\eta}\, (\Lambda_0^+)^{stab} \sim_{Leg.} \Lambda_0
\end{equation}
that fixes some non-empty open subset of $\Lambda_0$, as well as a Legendrian isotopy
\begin{equation}
\label{eq:two}
(\Lambda_{ 0})^{stab} \,\sharp_{\eta}\, \Lambda_0^+ \sim_{Leg.} \Lambda_0^+
\end{equation}
that fixes some non-empty open subset of $\Lambda_0^+$. The key point is that the two standard unknots $(\Lambda_0)^{stab}$ and $(\Lambda_0^+)^{stab}$ are contained inside disjoint Darboux balls (that live inside arbitrarily small neighbourhoods of the two disjoint overtwisted discs bounded by $\Lambda_0$ and $\Lambda_0^+$, respectively ). The sought isotopies are then constructed by Lemma \ref{lma:stabilisedunknot}.

Recall that the standard Legendrian unknot $(\Lambda_0^+)^{stab}$ admits an orientation-reversing contact isotopy inside the Darboux ball (it is induced by the contact lift of a rotation of the $(p,q)$-coordinates). Let $\eta'$ be the image of the Legendrian arc $\eta$ under this contact isotopy. For an oriented Legendrian $\Lambda$, we denote by $\overline{\Lambda}$ the same Legendrian but with the opposite orientation.

The Legendrian $\Lambda_0$ is Legendrian isotopic to $\Lambda_0 \,\sharp_{\eta}\, (\Lambda_0^+)^{stab}$ by \eqref{eq:one} above. The latter admits a Legendrian isotopy to $\Lambda_0 \,\sharp_{\eta^-}\, \overline{(\Lambda_0^+)^{stab}}$ where the orientation of the second summand has been reversed. Here the arc $\eta^-$ is the image of the arc $\eta$ under the isotopy, and we may assume that the equality $\eta^-=\eta$ holds in a neighbourhood of the initial overtwisted disc with boundary $\Lambda_0$. 

Since the stabilisation of the latter summand in $\Lambda_0 \,\sharp_{\eta-}\, \overline{(\Lambda_0^+)^{stab}}$ can be assumed to be very small and performed away from the arc $\eta^-$, it is thus possible to move it across the handle in the surgery region to the first summand. In other words: the latter Legendrian is Legendrian isotopic to $\Lambda_0^{stab} \,\sharp_{\eta^-}\, \overline{\Lambda_0^+}$, where the stabilisation has been performed on the first summand. By \eqref{eq:two} above, the latter is, in turn, Legendrian isotopic to $\overline{\Lambda_0^+}$. (Here we have used the property that $\eta^-$ coincides with $\eta$ near $\Lambda_0$.) 

It is finally just a matter of moving $\overline{\Lambda_0^+}$ back to $\overline{\Lambda_0}$ by time-reversing the original displacement of the overtwisted disc. Thus we have managed to create the sought orientation reversing Legendrian isotopy of $\Lambda_0$.
\end{proof}

\section{Proof of Theorem \ref{thm:main}}

The symplectisations induced by two different choices of contact forms are exact symplectomorphic; we are thus free to make a choice a contact form that fits our purposes. We start by fixing a three-dimensional overtwisted contact three-manifold $(Y,\alpha)$ endowed with some choice of contact form.

\subsection{Exact embeddings of tori}
\label{sec:torus}

Let $\Lambda \subset (Y,\alpha)$ be a Legendrian knot whose complement is overtwisted. Consider a closed embedding
$$\gamma_\lambda \colon S^1 \hookrightarrow \R^2_{(x_1,x_2)}$$
of a curve depending on $\lambda \ge 0$ whose image:
\begin{itemize}
\item is contained inside the subset $[-1,\lambda+1]\times[0,\epsilon]$; and
\item coincides with the two arcs $\{x_2=0,\epsilon\}$ inside the subset $\{x_1 \in [0,\lambda]\}$.
\end{itemize}
Note Lemma \ref{lma:immersion} applied to the closed curve $\gamma_\lambda$ and the Legendrian $\Lambda$ produces a Lagrangian immersion
$$L_\lambda \coloneqq L_{\Lambda,\gamma_\lambda} \subset S(Y,\alpha)$$
of the torus $\gamma_\lambda \times \Lambda\cong \T^2$. Moreover, if we take $\epsilon>0$ in the construction to be strictly smaller than the length of any Reeb chord on $\Lambda$, then the torus produced by the lemma is in fact embedded.

Recall that $L_\lambda$ can be parametrised by
$$ (\theta_1,\theta_2) \mapsto \left(\gamma_1(\theta_2),\phi^{\gamma_2(\theta_2)}_{R_\alpha}\left(\Lambda(\theta_1)\right)\right),$$
where $\gamma(\theta_2)=(\gamma_1(\theta_2),\gamma_2(\theta_2))$ is a parametrisation of $\gamma$ and $\Lambda(\theta_1)$ is a parametrisation of $\Lambda$.
\begin{rmk} The above Lagrangian torus is never exact. However, since $\alpha$ vanishes along any curve of the form $\eta=\{\theta_2=c\}$ (this follows from the property that $\Lambda$ is Legendrian), we conclude that the symplectic action vanishes along curves of this type, i.e.~$\int_\eta e^\tau\alpha=0$.
\end{rmk}

The next step is deforming the Lagrangian torus $L_\lambda$ to an exact Lagrangian embedding, without changing its Lagrangian isotopy class.

Let $U\subset Y$ be a standard jet-neighbourhood of $\Lambda$ that contains $\phi^{[-\epsilon,\epsilon]}_{R_\alpha}(\Lambda)$ (recall that $\epsilon>0$ is small). Consider the Legendrian boundary $\Lambda_0 \subset Y \setminus U$ of an overtwisted disc contained inside the complement $Y \setminus U$ of this jet-neighbourhood. This is possible since $Y\setminus \Lambda$ is assumed to be overtwisted.

Then, consider the Legendrian knot produced by the cusp connected sum
$$\Lambda' \coloneqq (\Lambda_0)^{stab} \,\sharp_\eta\, \phi^{-\epsilon}_{R_\alpha}(\Lambda)$$
of $(\Lambda_0)^{stab} \subset Y \setminus U$, i.e.~the stabilisation of the boundary $\Lambda_0$ of the overtwisted disc , and the Legendrian $\phi^{-\epsilon}_{R_\alpha}(\Lambda)$ for a suitable Legendrian arc $\eta$ with boundary on the two Legendrians; see Subsection \ref{sec:destabilising} for more details. The Legendrian arc $\eta$ used to construct the cusp connected sum that connects the two knots may be assumed to be disjoint from a smaller neighbourhood $V \subset U$ that contains $\phi^{[0,\epsilon]}_{R_\alpha}(\Lambda)$. In particular, the cusp connected sum $\Lambda'$ is disjoint from $\phi^{[0,\epsilon]}_{R_\alpha}(\Lambda)$. Since the negative Reeb flow induces a Legendrian isotopy from $\Lambda$ to $\phi^{-\epsilon}_{R_\alpha}(\Lambda)$ inside $Y \setminus \phi^\epsilon_{R_\alpha}(\Lambda)$ , for a suitable Legendrian arc $\eta$, we conclude from Lemma \ref{lma:stabilisedunknot} that 
\begin{lma}
\label{lma:Lambda'}
There exists a Legendrian isotopy that takes $\Lambda$ to $\Lambda'$ and which is supported inside $Y \setminus \phi^\epsilon_{R_\alpha}(\Lambda)$.
\end{lma}

We then allude to Chantraine's construction of Lagrangian trace cobordisms from a Legendrian isotopies, in order to construct the following Lagrangian cylinders.

\begin{lma}
There exist constants $A_0<B_0<A_1<B_1$ and exact Lagrangian cobordisms
$$C_i \subset S\left(Y \setminus \phi^\epsilon_{R_\alpha}(\Lambda),\alpha\right), \:i=0,1,$$
diffeomorphic to cylinders, where
\begin{itemize}
\item $C_0 \cap \{\tau \le A_0\}=(-\infty,A_0] \times \Lambda$,
\item $C_0 \cap \{\tau \ge B_0\}=[B_0,+\infty) \times \Lambda'$, 
 \item $C_1 \cap \{\tau \le A_1\}=(-\infty,A_1] \times \Lambda'$, and
\item $C_1 \cap \{\tau \ge B_1\}=[B_1,+\infty) \times \Lambda$.
\end{itemize}
In particular, $C_i$ is cylindrical outside of the subset $\{\tau \in [A_i,B_i]\}$. Without loss of generality, we may assume that $A_0=0$, $B_0 \gg 0$, $A_1=B_0+1$, and $B_1 \gg A_1$.

Finally, after the latter translation, we may assume that the concatenation
$$C_0 \odot C_1 \coloneqq (C_0 \cap \{\tau \le B_0\}) \cup (C_1 \cup \{\tau \ge B_0\})$$
is compactly supported Hamiltonian isotopic to $\R \times \Lambda$ inside $S(Y\setminus \phi^\epsilon_{R_\alpha}(\Lambda),\alpha)$.
\end{lma}
\begin{proof}
In \cite[Theorem 1.2]{Chantraine:Concordance} a Lagrangian trace cobordism in the symplectisation was constructed from a Legendrian isotopy. We take $C_0$ to be the Lagrangian trace cobordism produced by the isotopy from $\Lambda$ to $\Lambda'$ in $Y\setminus \phi^\epsilon_{R_\alpha}(\Lambda)$ provided by Lemma \ref{lma:Lambda'}. Then, we let $C_1$ be the Lagrangian trace cobordism of the Legendrian isotopy obtained by reversing the time in the same isotopy.

The prescribed cylindrical behaviour of $C_i$ can be achieved after appropriate translations of $C_i$ in the $\R_\tau$-factor.

The fact that the concatenations of the trace cobordisms are compactly supported Hamiltonian isotopic to the trivial cylinder was shown in e.g.~\cite[Proposition B.1]{DRS3}.
\end{proof}

Setting $\lambda \ge B_1$ in the construction of $L_\lambda$, we can replace the trivial cylinder
$$[0,B_1] \times \Lambda\subset L_\lambda$$
by the concatenation $C_0 \odot C_1$ of $C_0$ with $C_1$ produced by the above lemma. This means that we replace $[A_i,B_i] \times \Lambda$ by $C_i$ inside the subset $\tau \in [A_i,B_i]$ and, in particular, $[B_0,A_1] \times \Lambda$ is replaced by $[B_0,A_1]\times\Lambda'$ (recall that $A_1=B_0+1)$. We denote this deformed Lagrangian torus by $L_\lambda'$; note that this torus is still is embedded, but not necessarily exact. In addition, by the above lemma, it is Lagrangian isotopic to the original torus $L_\lambda$.

Exactness will be achieved by, roughly speaking, deforming the cylindrical part of $L_\lambda'$ given by
$$[B_0, A_1 ] \times \phi^\epsilon_{R_\alpha}(\Lambda) \subset L_\lambda'$$
by using the Reeb flow of a different contact form $\beta$. (See Proposition \ref{prp:deformation} below.) First we need to construct this new contact form.

\begin{lma}
The contact form $\alpha$ can be deformed outside of a neighbourhood of $\phi^{\epsilon}_{R_\alpha}(\Lambda)$ to a new contact form $\beta$ for which $\phi^\R_{R_{\beta}}(\phi^{\epsilon}_{R_\alpha}(\Lambda))$ is disjoint from a neighbourhood of $\Lambda'$. 
\end{lma}
\begin{proof}
The Legendrian $\Lambda'$ is a stabilisation of the cusp connected sum
$$\Lambda_0 \,\sharp_\eta\, \Lambda \subset Y \setminus \phi^{\epsilon}_{R_\alpha}(\Lambda),$$ and may thus be assumed to be contained in a small standard jet-neighbourhood $J^1(\Lambda_0 \,\sharp_\eta\,\Lambda)$ of $\Lambda_0 \,\sharp_\eta\, \Lambda$ that is contained inside $Y \setminus \phi^{\epsilon}_{R_\alpha}(\Lambda)$. Further, $\Lambda'$ may be identified with a stabilisation of the zero-section $j^10 \subset J^1(\Lambda_0 \,\sharp_\eta\,\Lambda)$ inside the same neighbourhood. By Lemma \ref{lma:posloop} we can assume that $\Lambda'$ is contained inside a neighbourhood of the form
$$\left\{ p \in (\delta_0,\delta), z \in [-\delta,\delta] \right\} \subset J^1(\Lambda_0 \,\sharp\, \Lambda)$$
for some $\delta>0$ sufficiently small and $0<\delta_0<\delta$. The fact that the latter neighbourhood still is disjoint from $\phi^{\epsilon}_{R_\alpha}(\Lambda)$ is important.

Recall that the above neighbourhood of $\Lambda'$ admits a periodic contact isotopy that is generated by a strictly positive autonomous contact Hamiltonian (see Lemma \ref{lma:posloop}). By Lemma \ref{lma:reeb} this contact isotopy is the Reeb flow for a contact form $\beta$. In particular, the image $\phi^\R_{R_{\beta}}(\phi^{\epsilon}_{R_\alpha}(\Lambda))$ is disjoint from the above neighbourhood of $\Lambda'$ for this choice of contact form.
\end{proof}

The deformation of $L_\lambda'$ to an exact Lagrangian embedding is finally produced by the below general proposition. 

\begin{prp}
\label{prp:deformation}
Let $L \subset S(Y,\alpha)$ be a closed Lagrangian submanifold that satisfies the following:
\begin{itemize}
\item
The intersection 
$$ L \cap \left\{\tau \in [\tau_0-\epsilon,\tau_0+\epsilon]\right\}=[\tau_0-\epsilon,\tau_0+\epsilon] \times (\Lambda^0 \sqcup \Lambda^1) $$
is cylindrical, where $\Lambda^i \subset Y$, $i=0,1,$ are two connected Legendrian submanifolds; and
\item The symplectic action vanishes on any element in $H_1\left(L \setminus \{\tau=\tau_0\}\right)$.
\end{itemize}
If there are no Reeb chords between $\Lambda^0$ and $\Lambda^1$ for a second contact form $\beta=e^f\alpha$, $f \colon Y \to \R$, then $L$ is Lagrangian isotopic to an exact Lagrangian embedding by an isotopy supported inside the subset $\{ \tau \ge \tau_0-\epsilon\}$.
\end{prp}
\begin{proof}
	We must first deform the Lagrangian $L$ in order to make it cylindrical in some subset of the symplectisation $S(Y,\beta)$ for the new contact form $\beta$.
	
\emph{Step 1 (Stretch the neck):}	
We stretch the neck of $L$ to $L_t$ by translating $L \cap \{\tau_0+\epsilon\}$ by $t \ge 0$ in the positive $\tau$-direction, and then extending the cylindrical part to
$$ [\tau_0-\epsilon,\tau_0+\epsilon+t] \times (\Lambda^0 \sqcup \Lambda^1).$$
This Lagrangian isotopy clearly produces a Lagrangian $L_t$ that still satisfies the same assumptions as the original Lagrangian $L$.

For $t \gg 0$, the image of the cylindrical subset
$$L_t \cap \left\{\tau \in [\tau_0-\epsilon,\tau_0+\epsilon+t]\right\} =[\tau_0-\epsilon,\tau_0+\epsilon+t] \times (\Lambda^0 \sqcup \Lambda^1)$$
under the exact symplectomorphism
\begin{gather*}
	\Phi \colon S(Y,\alpha) \to S(Y,\beta),\\
	(\tau,y) \mapsto (\tau-f,y),
\end{gather*}
 still contains a cylindrical subset. More precisely, we can assume that 
$$ \Phi(L_t) \cap \left\{\tau \in [A,B]\right\}=[A,B] \times (\Lambda^0 \sqcup \Lambda^1)$$
holds whenever $A<B$ satisfies $A\ge \tau_0-\epsilon+\min f$ and $B \le \tau_0+\epsilon+t -\max f$, where we have taken $t \gg 0$. 
 
\emph{Step 2 (Deform by using the Reeb flow of $\beta$):}
	
The non-existence of Reeb chords between $\Lambda^0$ and $\Lambda^1$ for the contact form $\beta$ gives us the following freedom for deforming $\Phi(L_t)$ without introducing double points. We can replace the cylinder
$$ [A,B] \times \Lambda^0 $$
with a cylinder that is constructed by an application of Lemma \ref{lma:immersion}, for \emph{any} properly embedded connected interval $\gamma \subset (A,B) \times \R$ that coincides with $[A,B] \times \{0\}$ outside of a compact subset, and which is graphical over the first coordinate; the assumption on the non-existence of Reeb chords implies that the new Lagrangian remains \emph{embedded} (and, for the same reason, it is Lagrangian isotopic to $\Phi(L_t)$).

It readily follows by the symplectic action computation in Lemma \ref{lma:immersion} that a suitable such curve $\gamma$ will make the resulting deformation of $L$ into to a Lagrangian embedding that is exact. One can simply take a curve that gives rise to a Lagrangian cylinder whose potential difference
$$ \int_\gamma e^{x_1}dx_2 \in \R$$
of its two boundary components cancels the potential difference of the value of $g$ on the two boundary components of
$$ \tilde{L} \coloneqq \Phi(L_t) \setminus ((A,B) \times \Lambda^0) $$
where $dg=\left.e^\tau\beta\right|_{T\tilde{L}}$. Here, note that each half
$$L_t^\pm \coloneqq L_t \cap \{ \pm \tau \le \tau_0\}$$
is an exact Lagrangian cylinder, in the sense that there is a globally defined primitive of $e^\tau \beta|_{TL_t^\pm}$ which is locally constant (but not globally constant) on the boundary $\partial L_t^\pm$.
\end{proof}

This also finishes the proof of Theorem \ref{thm:main}. \qed

\subsection{Exact embedding of a Klein bottle}
\label{sec:klein}

We start with the non-exact Lagrangian embedding
$$L_\lambda=L_{\Lambda_0,\gamma_\lambda} \subset S(Y,\alpha)$$
of a torus constructed using Lemma \ref{lma:immersion} in Subsection \ref{sec:torus} in the particular case when the Legendrian $\Lambda=\Lambda_0$ is the boundary of an overtwisted disc inside the overtwisted contact three-manifold $(Y,\alpha)$.

The next step is to modify $L_\lambda$ inside the subset $\{\tau \le 0\} \subset S(Y,\alpha)$ in order to to produce a Klein bottle. Recall that
$$C_0 \coloneqq L_\lambda \cap \{\tau \le 0\}$$
is a Lagrangian cylinder with Legendrian boundary equal to
$$\Lambda_0 \cup \phi^\epsilon(\Lambda_0) \subset \{0\} \times Y.$$
We will replace this cylinder with a Lagrangian cylinder $C \subset \{\tau \le 0\}$ that coincides with $C_0$ near the boundary $\{0\} \times Y$, and which reverses the orientation of precisely one of the boundary components when compared to $C_0$. If we replace the cylinder $C_0 \subset L_\lambda$ by $C$, then the new Lagrangian $L_\lambda'$ that we obtain is an embedded Klein bottle that coincides with $L_\lambda$ in the subset $\{\tau \ge 0\}$. Once this has been established, the deformation of this Klein bottle needed in order to turn it into an exact Lagrangian can be carried out by the very same argument as in the case of a torus, which we treated in Subsection \ref{sec:torus}.

What remains is to describe the construction of the cylinder $C$. Start with the cylinder $C_0$ with boundary is given by $\Lambda_0 \cup \phi_{R_\alpha}^\epsilon(\Lambda_0)$. The sought deformation $C$ of $C_0$ can be taken to be the concatenation (and translation) of $C_0$ and the Lagrangian trace cobordism of a Legendrian isotopy that
\begin{itemize}
\item fixes one Legendrian boundary component of $C_0$ pointwise, while it
\item reverses the orientation of the second Legendrian boundary component;
\end{itemize} Again we refer to see \cite{Chantraine:Concordance} for the construction of the Lagrangian trace cobordism. The sought orientation reversing Legendrian isotopy exists by Lemma \ref{lma:reversing}, since $\phi_{R_\alpha}^\epsilon(\Lambda_0)$ is the boundary of an overtwisted disc that is contained in the complement of $\Lambda_0$.

\subsection{Exact embedding of a torus with vanishing Maslov class}

Assume that the first Chern class of $(Y,\ker \alpha)$ vanishes. If we take $\Lambda$ to be null homologous and satisfy $\OP{rot}(\Lambda)=0$, then the torus $L_\lambda=L_{\Lambda,\gamma_\lambda} \subset S(Y,\alpha)$ above admits a curve of Maslov index equal to two. More precisely, the Maslov class vanishes on any curve $\{\gamma_1(\theta_0)\} \times \Lambda \subset L_\lambda$, while it takes the value $\pm2$ on any curve of the form
$$\theta \mapsto \left(\gamma_1(\theta),\phi^{\gamma_2(\theta)}_{R_\alpha}(p)\right), \:\: \text{ for some } p \in \Lambda.$$
Hence the construction of the exact Lagrangian torus needs to be deformed in order for it to yield an exact Lagrangian torus of vanishing Maslov class.

We perform a construction similar to the one in Subsection \ref{sec:klein}, i.e. we deform the Lagrangian cylinder
$$C_0 \coloneqq L_\lambda \cap \{\tau \le 0\}$$
by inserting a non-trivial Lagrangian cylinder that is induced by the trace of a non-trivial Legendrian isotopy of $\Lambda$. Then we carry out the deformation by a Lagrangian isotopy as in Subsection \ref{sec:torus} needed in order to achieve exactness.

Instead of the trace of an orientation-reversing isotopy of $\Lambda$ which was used for the construction of the Klein bottle, we need a Legendrian isotopy that starts and ends at $\Lambda$, and which induces a difference of $-2$ of the value of its Maslov potentials when comparing the negative and positive ends of the Lagrangian trace cobordism produced by \cite{Chantraine:Concordance}. See e.g.~\cite[Section 2.5]{Dimitroglou:Families} for the definition of the Maslov potential in this setting. Such a loop of Legendrians can indeed be constructed inside $Y \setminus \phi^\epsilon_{R_\alpha}(\Lambda)$, by relying on the overtwistedness of this contact manifold. Namely, since $\Lambda \subset Y \setminus \phi^\epsilon_{R_\alpha}(\Lambda)$ is stabilised, the loop produced by Lemma \ref{lma:posloop} can readily be seen to have the sought properties concerning the induced shift in the Maslov potential. To see this, we note that the once stabilised zero section $(j^10)^{stab}$ has rotation number ${\tt rot}=\pm 1$ compared to the canonical trivialisation of $\ker \alpha_0 \subset T(J^1S^1)$.
\qed

\section{Proof of Theorem \ref{thm:noninterlinkedness}}
Write $Y=J^1S^1$. Further, let $\theta',p',z'$ denote the coordinates on the standard contact jet-neighbourhood $J^1\Lambda'$ of $\Lambda' \coloneqq j^1\cos{\theta} \subset Y$ in which the Legendrian $\Lambda'$ corresponds to the zero section $j^10=\{z'=p'=0\} \subset J^1\Lambda'$; see \cite[Theorem 6.2.2]{Geiges}. The contact isotopy $\phi^t$ defined by $\theta \mapsto \theta-t$ fixes the neighbourhood
$$U \coloneqq \{p' \in [\delta_0,\delta], z' \in [-\delta,\delta]\}$$
for $0<\delta_0<\delta$. One immediately verifies that its corresponding contact Hamiltonian is positive and autonomous there. We can extend this generating Hamiltonian to all of $Y$ and moreover assume that it is autonomous and satisfies
$$H\colon Y \to [c,+\infty)$$
for some $c>0$. Again, denote by $\phi^t \colon Y \to Y$ the induced contact isotopy, which may be assumed to be well-defined for all $t \in \R$.

The Legendrian $\Lambda$ given as the stabilisation of $\Lambda'$ lives in the standard contact jet-neighbourhood of $\Lambda'$ by construction. After a Legendrian isotopy supported inside the same neighbourhood, it can be placed entirely in the subset $U \subset J^1\Lambda'$ described above; see Lemma \ref{lma:posloop}. Since $\phi^t$ preserves the neighbourhood $U$, and since the original Legendrian $j^10 \subset Y=J^1S^1$ is disjoint from $U$, it thus follows that $\phi^t(j^10) \cap \Lambda = \emptyset$ for all $t \in \R$, as required.
\qed

\bibliographystyle{alpha}
\bibliography{refs} 

\end{document}